\documentclass[reqno]{amsart}

\usepackage{amsmath,amssymb,amsthm}
\usepackage{graphicx,tikz}
\usetikzlibrary{shadings}

\newtheorem*{thm}{Theorem}
\newtheorem{proposition}{Proposition}
\newtheorem{corollary}{Corollary}

\newtheorem*{lemma}{Lemma}

\theoremstyle{definition}

\theoremstyle{remark}

\begin{document}

\title[]{Distance Equilibrium Measures and\\ curvature in metric spaces}
 \thanks{}

\author[]{Stefan Steinerberger}
\address{Department of Mathematics and Department of Applied Mathematics, University of Washington, Seattle, WA 98195, USA}
\email{steinerb@uw.edu}

\begin{abstract} Let $(X,d)$ be a compact metric space. We consider the behavior of probability measures $\mu$ with the property that
$$  \int_{X} d(x, y) d\mu(y)  \qquad \mbox{is independent of}~x \in X.$$
It appears that such measures, when they exist, encode a `curvature-type' quantity. 
We investigate this in the special case where $X$ is a closed, convex curve in $\mathbb{R}^2$ and $d = \| \cdot \|_2$ is the Euclidean distance: even a single point with small curvature implies non-existence of such a measure. Conversely, such a measure $\mu$ exists for all curves whose curvature is sufficiently close to constant. Curvature is usually defined by second derivatives; this one is defined via an integral equation which makes sense in much rougher spaces.  Connections to curvature on graphs, the Gross-Stadje Theorem and magnitude are discussed.
\end{abstract}

\maketitle

\section{Introduction and Results}
\subsection{Introduction.}
Let $(X,d)$ be a compact metric space. We consider probability measures $\mu$ with the property that 
$$ f(x) = \int_{X} d(x, y) d\mu(y) \qquad \mbox{is constant for all}~x \in X.$$
The main purpose of this paper is to suggest that, whenever such a measure $\mu$ exist, it may serve as a meaningful notion of curvature. $\mu$ satisfies an integral equation instead of being defined via differentiation; this allows for this notion to be implemented in much rougher spaces. This object has already been considered \cite{steinerberger} in the case when $(X,d)$ is a finite, combinatorial graph $G=(V,E)$, see \S 2.1.  The main purpose of this paper is to consider the case where $X \subset \mathbb{R}^2$ is a bounded subset with the usual Euclidean distance $d(x,y) = \|x-y\|_{\ell_2}$. We will see that a very natural special case is that of $X$ being a strictly convex curve.

\begin{center}
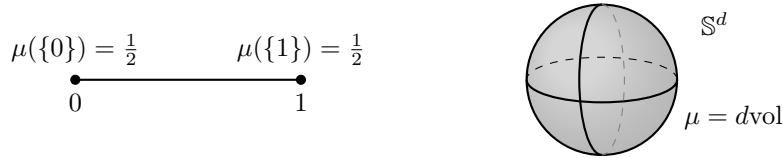
\begin{figure}[h!]
\begin{tikzpicture}
\filldraw (0,0) circle (0.06cm);
\draw [thick] (0,0) -- (3,0);
\filldraw (3,0) circle (0.06cm);
\node at (0, -0.3) {$0$};
\node at (3, -0.3) {$1$};
\node at (0, 0.4) {$\mu(\left\{0 \right\}) = \frac{1}{2}$};
\node at (3, 0.4) {$\mu(\left\{1 \right\}) = \frac{1}{2}$};
  \fill[ball color=white, opacity=0.2] (7,0) circle (1);
  \draw[thick] (7,0) circle (1);
  \draw[dashed, black] (8,0) arc (0:180:1 and 0.3);
  \draw[thick, black] (6,0) arc (180:360:1 and 0.3);
  \draw[dashed, gray] (7,1) arc (90:-90:0.3 and 1);
  \draw[thick, black] (7,1) arc (90:270:0.3 and 1);
 \node at (8.5, 0.75) {$\mathbb{S}^{d}$};
  \node at (8.75, -0.5) {$\mu = d\mbox{vol}$};
\end{tikzpicture}
\caption{A measure $\mu$ on the unit interval $X = [0,1]$ (left) and the sphere (right) solving the integral equation.}
\label{fig:flat}
\end{figure}
\end{center}

The case of a line segment (in any dimension) is shown in Fig. \ref{fig:flat}. The measure $\mu$ that puts half of the mass on each endpoint solves the equation; with a bit of imagination, this corresponds to our intuitive understanding: the endpoints have `infinite curvature', the interior is flat. The case of the sphere $\mathbb{S}^{d}$ has an easy solution: the (normalized) surface measure $\mu = d\sigma$ is an admissible choice. This is also consistent with the idea that curvature on $\mathbb{S}^{d}$ should be constant. However, integral equations are quite delicate; it is a compact operator and there will not be a solution for \textit{most} right-hand sides. Moreover, the inversion problem can be highly nontrivial and unstable. Indeed, for most sets $X$ no such measure exists.

\begin{proposition}[Wilson \cite{cleary}]
Let $X \subset \mathbb{R}^n$ be compact and connected. If there exists a probability measure $\mu$ on $X$ such that
$$ \int_{X} \| x- y\| d\mu(y) \qquad \mbox{is constant for all}~x\in X,$$
then $X$ is either a line segment or does not contain three points on any line.
\end{proposition}
Numerically, if one discretizes a set $X$ and then solves the corresponding linear system, the results are actually somewhat convincing even though the limiting object might not exist (see \S 1.5 for examples). In terms of understanding the integral equation, Proposition 1 suggests a natural class of sets to be investigated: we assume $X \subset \mathbb{R}^2$ is a strictly convex closed curve. 
The example $h(\theta) = 1 + 0.1 \cos{(3 \theta)}$ is shown in Fig. \ref{fig:2}. Numerically, the solution of the (discretized) integral equation corresponds to a smooth density on the curve that behaves like curvature.

\vspace{-5pt}
\begin{center}
\begin{figure}[h!]
\begin{tikzpicture}
\node at (0,0) {\includegraphics[width=0.25\textwidth]{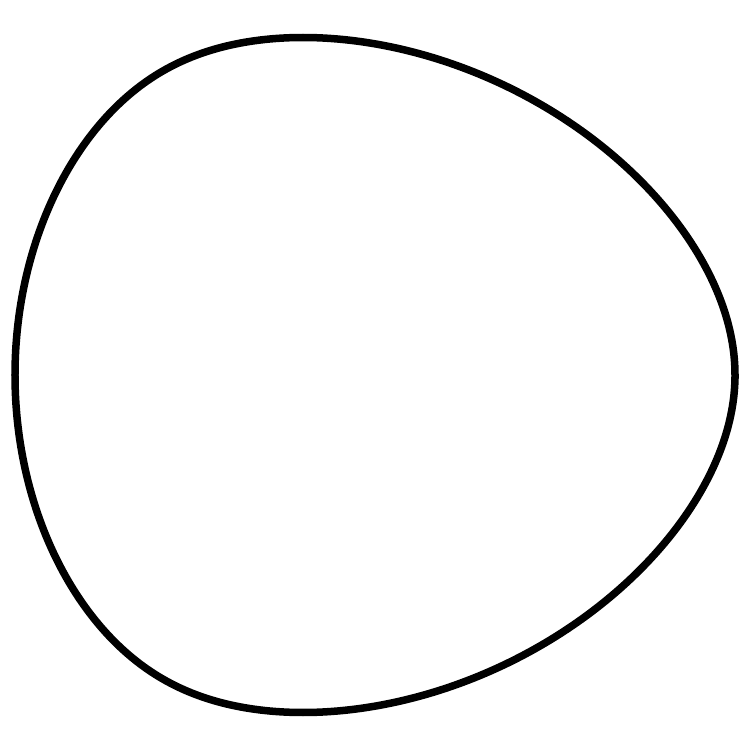}};
\node at (4.5,0) {\includegraphics[width=0.3\textwidth]{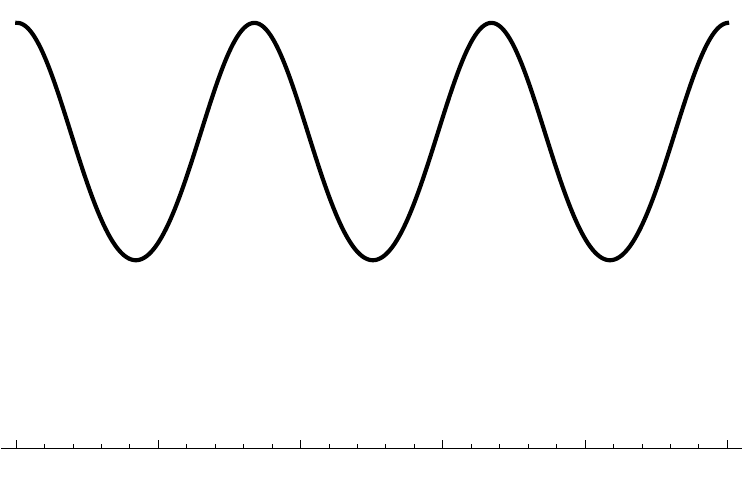}};
\draw [thick] (2.6,-1.03) -- (6.35,-1.03);
\draw [thick] (2.6, -1.1) -- (2.6, -0.95);
\draw [thick] (6.35, -1.1) -- (6.35, -0.95);
\node at (9,0) {\includegraphics[width=0.25\textwidth]{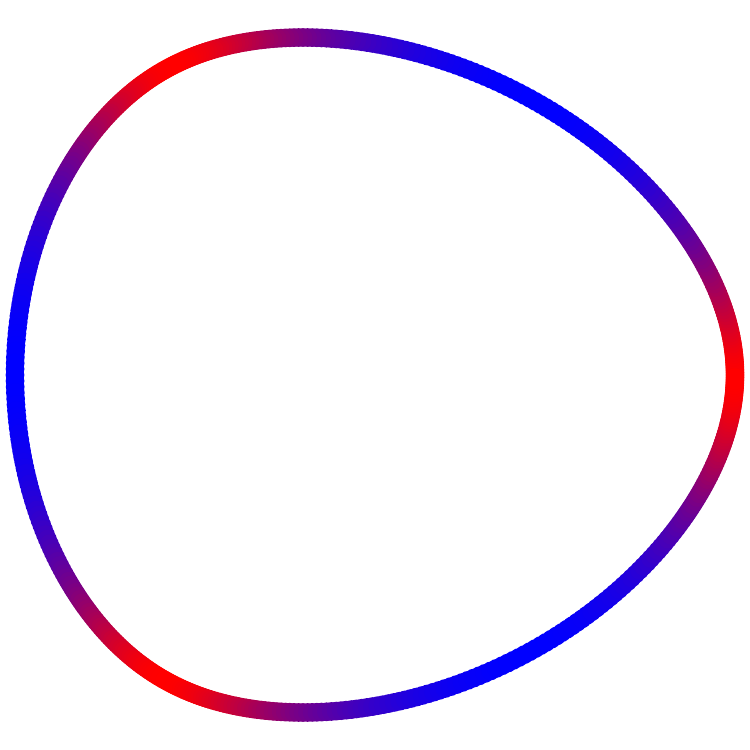}};
\end{tikzpicture}
\caption{Left: a convex curve $X \subset \mathbb{R}^2$. Middle: an approximation of the density of $\mu$ as a function of the arclength. Right: the curve colored by the density of $\mu$ (red means higher density).}
\label{fig:2}
\end{figure}
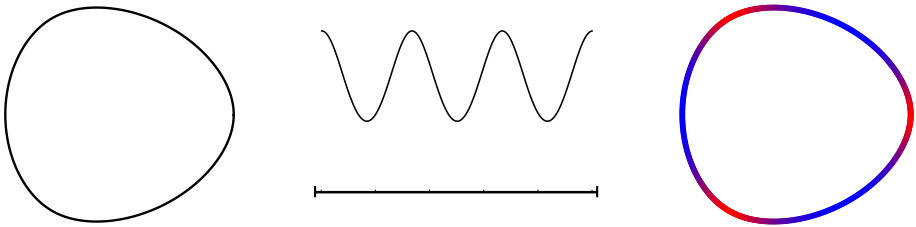
\end{center}
\vspace{-10pt}

\subsection{Small curvature implies non-existence.} For closed, convex curves in $\mathbb{R}^2$ Proposition 1 can be further refined: the relevant quantity is the curvature. 

\begin{proposition} Let $X$ be a smooth, closed, strictly convex curve in $\mathbb{R}^2$ satisfying
$ X \subset \left\{x \in \mathbb{R}^2: 0.999 \leq \|x\| \leq 1.001 \right\}$
and assume there exists a point on $X$ where the curvature is $\kappa \leq 0.0001$. Then, for every probability
measure $\mu$ on $X$, 
$$ \int_{X} \| x- y\| d\mu(y) \qquad \mbox{is } \emph{not} \mbox{ constant for all}~x\in X.$$
\end{proposition}

The proof is elementary and has not been optimized for constants. The assumption that $X$ is uniformly close to the unit disk is almost surely not necessary but simplifies the argument and helps to illustrate the main point: even for smooth, closed, convex curves in $\mathbb{R}^2$ that are close to unit circle, a single point with small curvature is sufficient to ensure that no such $\mu$ exists. It is clear from the proof that the argument extends to higher dimensions; however, good quantitative estimates on the `critical' curvature appears challenging even in two dimensions.

\subsection{An Existence Result}
Our main result is an existence result: if the curvature is close to constant, then this implies existence of a solution of the integral equation.

\begin{thm} There exists a universal $ \varepsilon_0 > 0$ such that if $X$ is a smooth, closed, strictly convex curve in $\mathbb{R}^2$ satisfying
$$1 - \varepsilon_0 \leq \frac{\emph{length}(X_n)}{2\pi }  \min_{x \in X_n} \kappa(x)   \leq  \frac{\emph{length}(X_n)}{2\pi }  \max_{x \in X_n} \kappa(x)  \leq 1 + \varepsilon_0,$$
then there exists a probability
measure $\mu$ on $X$ so that 
$$ \int_{X} \| x- y\| d\mu(y) \qquad \mbox{ is constant for all}~x\in X.$$
\end{thm}

The proof uses a limiting argument which leads to an unspecified constant $\varepsilon_0 > 0$.  It seems that one could use our argument to obtain an explicit value for $\varepsilon_0$, however,
that value would presumably be far from the optimal value. It would be desirable to understand the interplay between Proposition 2 (non-existence in the case of small curvature) and the main result (existence for sufficiently round curves) better.
 
 \vspace{-10pt}
 \begin{center}
\begin{figure}[h!]
\begin{tikzpicture}
\node at (0,0) {\includegraphics[width=0.3\textwidth]{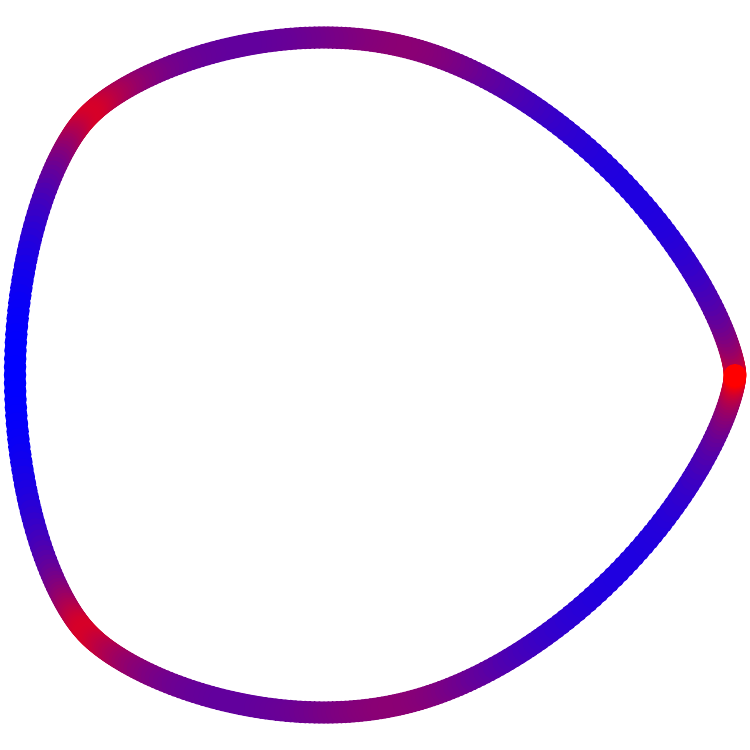}};
\node at (6,0) {\includegraphics[width=0.3\textwidth]{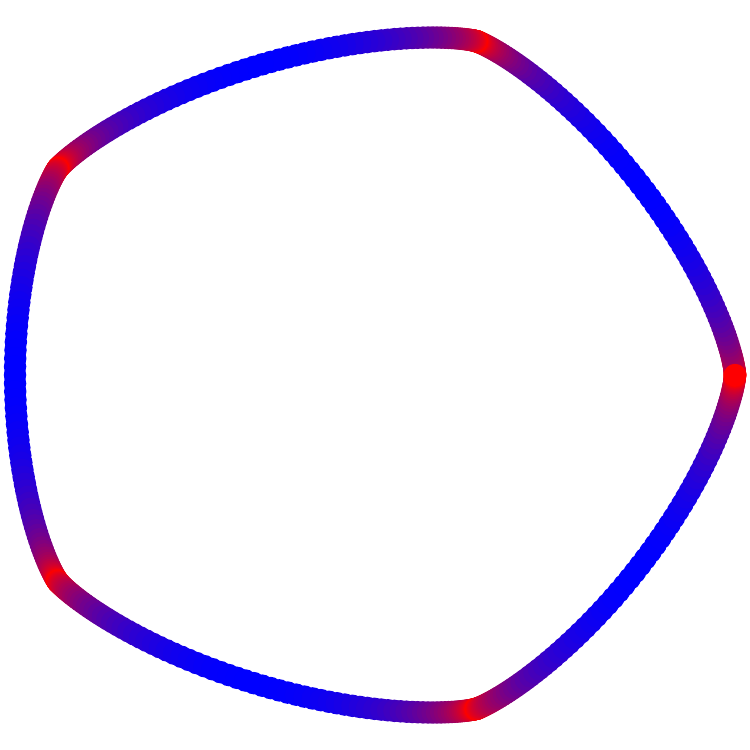}};
\end{tikzpicture}
\caption{Curves colored by the density of $\mu$ (red means higher density).}
\label{fig:3}
\end{figure}
\end{center}
\vspace{-10pt}

\subsection{Open problems.} This suggests a number of open problems.

\begin{enumerate} 
\item What is the precise existence theory for solutions of the integral equation when $X$ is a strictly convex curve in the plane? Curvature, in some form, will have to play an important role.
\item What can be said about the existence theory for solutions of the integral equation when $X = \partial \Omega \subset \mathbb{R}^n$ is the boundary of a smooth convex domain? The relevant quantity will presumably be Gaussian curvature.
\item Is there another way to establish the existence of a solution of the integral equation that is different from the one outlined in the proof of the Theorem? Different approaches may have different limitations and might be able to highlight other perspectives.
\item What about the more general setting of \textit{signed} measures as opposed to merely probability measures? It appears as if the case of probability measures corresponds to a type of `positive curvature' and convexity (see Proposition 1). Signed curvatures would provide a more general framework.
\item Are there other metric spaces where the equation can be analyzed? The results are very convincing when $(X,d)$ is a finite graph, see \S 2.1 and \cite{steinerberger}.
\end{enumerate}

One of the most curious aspects is perhaps the following: while the existence of a solution of the integral equation may be quite delicate, the discretization leads to a linear system of equations that will typically have a solution. Suppose, for example, we have a continuous, closed curve $\gamma:\mathbb{S}^1 \rightarrow \mathbb{R}^2$ and discretize it in $n$ points $p_1, \dots, p_n \in \mathbb{R}^2$ where $ p_j = \gamma( e^{2\pi i j/n} )$.  Then, accounting for the arclength element, this leads to the linear system 
$ Ax = b$
where 
$$ A_{k \ell} = \| p_k - p_{\ell}\| \cdot \| \gamma'( e^{2\pi i \ell/n} )\| \qquad \mbox{and} \qquad b_k = \| \gamma'( e^{2\pi i k/n} )\|.$$
An example of a non-convex curve is shown in Fig. \ref{fig:6}.  Even though the limiting integral equation has no solution (because the curve has 3 points on a line), the discretized linear system appears to have a very nice solution that mirrors the geometry of the curve: the signed measure is positive in the outer (`convex') regions of the curve and negative (`concave') in the inner parts.
\vspace{-5pt}
  \begin{center}
\begin{figure}[h!]
\begin{tikzpicture}
\node at (0,0) {\includegraphics[width=0.25\textwidth]{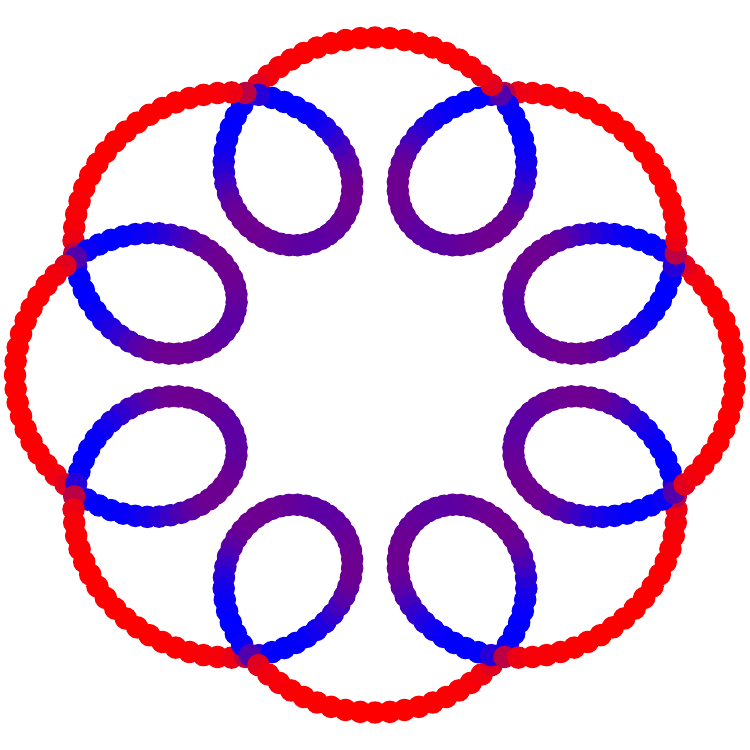}};
\node at (6,0) {\includegraphics[width=0.4\textwidth]{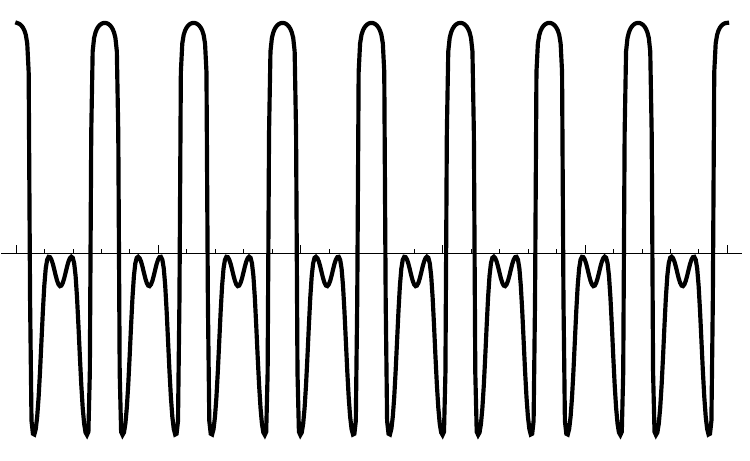}};
\end{tikzpicture}
\caption{A non-convex curve colored by density of the (signed) measure (left) and the density of the signed measure written as the numerical solution of a discretization (right).}
\label{fig:6}
\end{figure}
\end{center}
\vspace{-10pt}
This phenomenon persists when discretizing two-dimensional domains (see Fig. \ref{fig:7}). The arising solutions are not probability measures but they do appear to converge, at least in some weak sense, to a type of singular, signed measure. 
\vspace{-10pt}
  \begin{center}
\begin{figure}[h!]
\begin{tikzpicture}
\node at (0,0) {\includegraphics[width=0.46\textwidth]{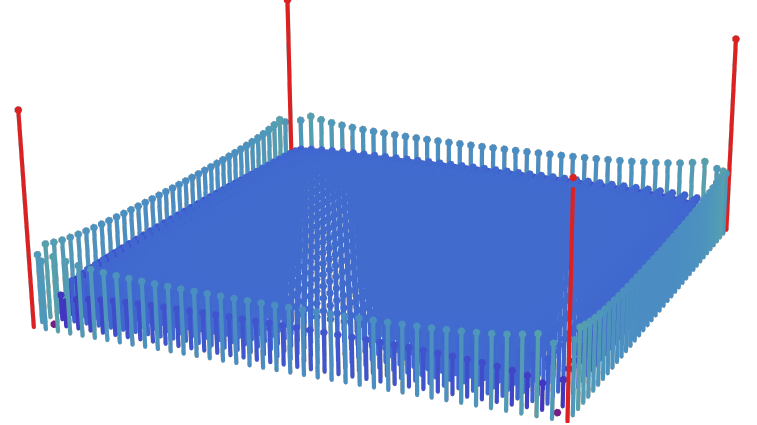}};
\node at (8,0) {\includegraphics[width=0.52\textwidth]{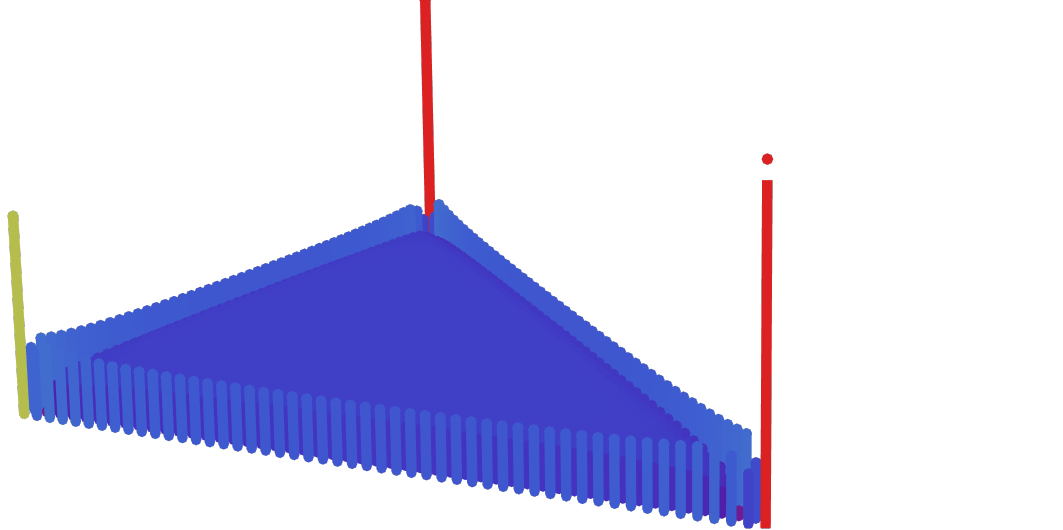}};
\end{tikzpicture}
\caption{Numerical solutions of a discretized square (left) and a discretized right-angled triangle (right). The numerical solutions are large and positive in corners, slightly less large at the boundary.}
\label{fig:7}
\end{figure}
\end{center}
   Fig. \ref{fig:7} highlights that the geometry of the measure relates to the underlying geometry: it identifies the boundary and, in the case of the right-angled triangle, even the opening angle. Given $n$ distinct points in $\mathbb{R}^d$, the distance matrix $ D_{ij} = \|x_i - x_j\|$ is invertible \cite{mc}. This suggests an existence theory for curves (by sampling equispaced with respect to arclength); the case of surfaces where the volume of a local Voronoi cell appears seems to be more delicate.

\section{Related topics and ideas}
Similar ideas have shown up in a number of different areas. We quickly summarize some of these ideas in the areas of (1) Graph Curvature, (2) Distance Geometry and the Gross-Stadje theorem, (3) some rigidity theorems in Differential Geometry and (4) Leinster's definition of Magnitude.

\subsection{Graph curvature.} Let $G=(V,E)$ be finite, connected graph. There has been a lot of recent interest in the question of whether it is possible to define a notion of `curvature' in the discrete setting. Examples include the Ollivier-Ricci curvature \cite{olli}, the Lin-Lu-Yau curvature \cite{lin} and the resistance curvature \cite{dev}.  A notion proposed in \cite{steinerberger} is as follows: if the vertices are given by $\left\{v_1, \dots, v_n\right\}$ and the distance between $v_i$ and $v_j$ is given by $d(v_i, v_j)$, then solve the linear system

$$ \begin{pmatrix}
d(v_1, v_1)  && d(v_1, v_2) && \cdots && d(v_1, v_n)\\
d(v_2, v_1) && d(v_2, v_2) && \cdots && d(v_2, v_n) \\
d(v_3, v_1) && d(v_3, v_2) && \cdots && d(v_3, v_n) \\
\cdots && \cdots && \cdots && \cdots \\
d(v_n, v_1) && d(v_n, v_2) && \cdots && d(v_n, v_n) \\
\end{pmatrix} \cdot 
\begin{pmatrix} x_1 \\ x_2 \\ x_3 \\ \cdots \\ x_n \end{pmatrix} =
\frac{1}{n} \begin{pmatrix} 1 \\1 \\ 1 \\ \cdots \\ 1 \end{pmatrix}
$$
\vspace{5pt}

and define the `curvature' in the vertex $v_i$ to be $x_i$. Some examples of what this looks like
are shown in Figure \ref{fig:calc}. This notion has been further explored by Chen-Tsui \cite{chen}, Cushing-Kamtue-Law-Liu-M\"unch-Peyerimhoff \cite{cushing}, Robertson \cite{robertson} and Robertson-Southerland-Surya \cite{rob2}. The same linear system with $d(\cdot, \cdot)$ replaced by the `resistance' (a random-walk based notion of distance) has inspired the idea of resistance curvature \cite{dev} which has many desirable properties as well. It would be interesting to understand whether other notions of distance on graphs may give rise to other interesting `curvature-like' definitions of this type.

\begin{center}
\begin{figure}[h!]
\begin{tikzpicture}
\node at (0.25,0) {\includegraphics[width=0.25\textwidth]{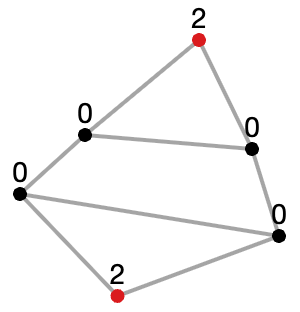}};
\node at (4,0) {\includegraphics[width=0.33\textwidth]{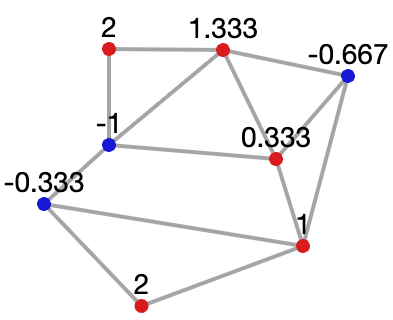}};
\node at (8.7,0) {\includegraphics[width=0.38\textwidth]{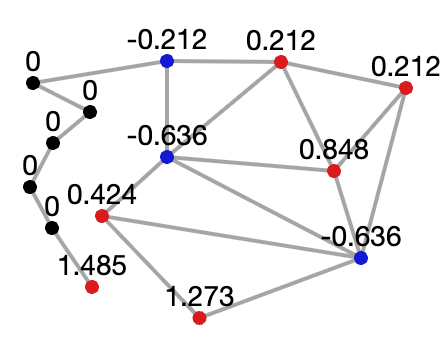}};
\end{tikzpicture}
\caption{Three graphs and their curvatures; pictures created using the `Graph Curvature Calculator' \cite{calc}.}
\label{fig:calc}
\end{figure}
\end{center}

 In contrast to the other existing notions of curvature on graphs, this definition does not seem to have a direct analogue in the continuous setting: it is global while one would expect a notion of curvature to be local. Nonetheless, for reasons not yet fully understood, it does seem to work exceedingly well on graphs. The question of a continuous analogue has inspired this paper; once more, we find that a globally defined object has a remarkable ability to capture local behavior.

\subsection{Gross-Stadje Theorem.} There is a remarkable result about compact, connected metric spaces due to Oliver Alfred Gross \cite{gross}.

\begin{thm}[Gross, 1964]
Let $(X,d)$ be a compact, connected metric space. There exists a unique number $r > 0$ such that for any $n \in \mathbb{N}$ and any collection of $n$ not necessarily distinct $x_1, \dots, x_n \in X$ there exists a point $x \in X$ such that
$$ \frac{1}{n} \sum_{i=1}^{n} d(x,x_i) = r.$$
\end{thm}

It is easy to see that $\mbox{diam}(X)/2 \leq r \leq \mbox{diam}(X)$ but few other universal results exists. We refer to a fantastic survey by Cleary-Morris-Yost \cite{cleary}. The value of $r(X,d)$ is not known except in a few cases but is easy to numerically approximate \cite{larcher}. One of the few known ways to compute the number $r(X,d)$ exactly is to find a probability measure $\mu$ on $X$ such that
$$ \int_{X} d(x,y) d\mu(y) = c \qquad \mbox{is a constant independent of}~x.$$
Choosing $x_1, \dots, x_n$ in a way that approximates the measure $\mu$ then shows that $c =r(X,d)$. This allows, for example, the computation of the number for $X = [0,1]$ or $X = \mathbb{S}^{d}$. Moreover, this is exactly the integral equation that we consider here.  To the best of our knowledge, the geometric information encoded in $\mu$ has not been related to curvature before. We refer to
Farkas-Revesz \cite{farkas0, farkas, farkas2},
Hinrichs \cite{h0, hinrichs},
Hinrichs-Nickolas-Wolf \cite{hnw},
Hinrichs-Wenzel \cite{hw},
Morris-Nickolas \cite{mn},
Nickolas-Wolf \cite{nw1},
Stadje \cite{stadje},
Thomassen \cite{thom} 
and Wolf \cite{wolf1, wolf2, wolf3, wolf4, wolf5, wolf6}.

\subsection{Differential Geometry.} 
One might study the optimization problem
$$ \int_{X \times X} d(x,y) d\mu(x) d\mu(y)  \rightarrow \max$$
over all probability measures with support in $X$ (under some conditions, the maximizer of this problem has the desired properties, see \S 3.1).  This leads to the discrete problem of finding $n$ points so as to maximize
$ \sum_{i,j=1}^{n} d(x_i, x_j) \rightarrow \max.$
This problem was studied by Alexander-Stolarsky \cite{alex}, Bj\"orck \cite{bjorck} and many others. When $X = \mathbb{C}$ and $d(x,y)$ is replaced by $\log| x- y|$, this leads to the notion of transfinite diameter, see Fekete \cite{fekete} and Fekete-Szeg\"o \cite{fsz}. In the setting of more abstract spaces,
Grove-Markvorsen \cite{gm} proved comparison and rigidity bounds in Alexandrov spaces relating the \textit{maximum mean spread} to the curvature of the underlying space. Kokkendorf \cite{kokkendorf} proved that among all $n-$dimensional Riemannian manifolds with $\mbox{Ric}(X) \geq n-1$, the maximum average distance is maximized by the $n-$dimensional sphere. In the setting of combinatorial graphs, these ideas have been related to universal inequalities for distance matrices \cite{steinmeasure} and measures that generate Lipschitz embedding of the graph into $\ell^1(\mathbb{R}^m)$ \cite{steinembedding}.

\subsection{Magnitude.} All these ideas seem to be related, in some way, to the notion of \textit{magnitude}, a numerical invariant of metric spaces proposed by Leinster \cite{lei1, lei2, lei3} that can be interpreted as a version of the Euler characteristic of a category. The precise type of connection is less clear, however, the following quote from a survey of Leinster-Meckes \cite{lei4} is unambiguous.

\begin{quote}
A \textbf{weight measure} on a compact metric space $(A,d)$ is a finite signed Borel measure $\mu$ on $A$ such that
$$ \int_A e^{-d(a,b)} d\mu(b) = 1$$
for every $a \in A$. [...] This suggests defining the magnitude of a compact, metric space to be $|A| = \mu(A)$ whenever $A$ possesses a weight measure $\mu$. (Leinster and Meckes, \cite{lei4})
\end{quote}

This weight measure is our integral equation with $d(x,y)$ replaced by $e^{-d(x,y)}$. Magnitude has since been extensively studied, we refer to Barcel\'o-Carbery \cite{bar}, Gimperlein-Goffeng \cite{gimp}, Meckes \cite{meckes} for results in Euclidean space. We emphasize that the notion of magnitude is a single number atttached to the metric space whereas we are interested in the underlying measure (provided it exists). One cannot help but wonder how all these different notions are connected.

\section{Proof of Proposition 1 and Proposition 2} 
\subsection{Proof of Proposition 1}
\begin{proof}[Proof summarized from \cite{cleary}]
Suppose $X$ contains three points on a line $\ell$ but $X$ itself is not a line segment. We take two points with maximal distance from each other on the line $\ell$ and call them $a$ and $b$.  The third point $c$ on the line lies in the convex hull of $a$ and $b$ and we have $c = \lambda a + (1-\lambda) b$ for some $0 < \lambda < 1$. Then
\begin{align*}
 x \notin \left\{ a, b \right\} \implies \| x - c\| &= \| \lambda x + (1-\lambda) x  - \lambda a - (1-\lambda) b\| \\ 
&< \lambda \|x-a\| + (1-\lambda) \|x-b\|.
 \end{align*}
If not all the measure is concentrated in $\left\{a,b\right\}$, then integrating this inequality over $X$ shows that
 $$ \int_{X} \|x-c\| d\mu(x) < \lambda \int_{X} \|x-a\| d\mu(x) + (1-\lambda) \int_{X} \|x-b\| d\mu(x)$$
 which contradicts the assumption that all three of these integrals attain the same value. If all the measure is supported in $\left\{a,b \right\}$, then it is easy to see that $X$ has to be a line segment between $a$ and $b$. The result follows.
\end{proof}

\subsection{Proof of Proposition 2}
\begin{proof}  The proof is by contradiction. Let $X$ be a smooth, closed, strictly convex curve in $\mathbb{R}^2$ contained in $X \subset \left\{x \in \mathbb{R}^2: 0.999 \leq \|x\| \leq 1.001 \right\}$. 
We assume that the curvature in a specific point is very small and the remainder of the argument is concerned with obtaining a contradiction. Having a point with small curvature means that there exists $c>0$ (depending on the curvature) such that for arbitrarily small $d>0$ one can find three points on the curve which after rotation and translation are arranged as in Fig. \ref{fig:prop2}. 

\begin{center}
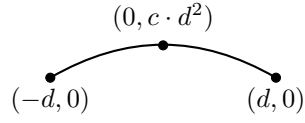
\begin{figure}[h!]
\begin{tikzpicture}
\filldraw (0,0) circle (0.06cm);
\filldraw (3,0) circle (0.06cm);
\draw [thick] (0,0) to[out=30, in=150] (3,0);
\node at (0, -0.3) {$(-d,0)$};
\node at (3, -0.3) {$(d,0)$};
\filldraw (1.5, 0.43) circle (0.06cm);
\node at (1.5, 0.8) {$(0,c \cdot d^2)$};
\end{tikzpicture}
\caption{A relatively flat segment of the convex curve.}
\label{fig:prop2}
\end{figure}
\end{center}

If such a probability measure $\mu$ were to exist, then one can deduce that
\begin{align*}
 I = \int_{X} (\| (-d,0) - y\| + \|(d,0) - y\| - 2 \|(0, c \cdot d^2) - y\|) ~d\mu(y) = 0.
 \end{align*}
 We consider, for $y_1, y_2 \in \mathbb{R}$ fixed, the real-valued function
 $$ f(t) = \sqrt{(y_1 - t)^2 + y_2^2}$$
 and note that
 \begin{align*}
  \| (-d,0) - y\| + \|(d,0) - y\| - 2 \|y\| = f(d) + f(-d) - 2f(0).
 \end{align*}
Using
 \begin{align*}
   f(d) + f(-d) - 2f(0) &\geq  d^2 \min_{|t| \leq d} f''(t) \\
   &= d^2 \min_{|t| \leq d} \frac{y_2^2}{((t-y_1)^2 + y_2^2)^{3/2}}\\
   &\geq d^2 \frac{y_2^2}{(d + |y_1|)^2 + y_2^2)^{3/2}}
 \end{align*}  
 one has 
 $$ \| (-d,0) - y\| + \|(d,0) - y\|  - 2 \|y\| \geq  \frac{y_2^2}{(d + |y_1|)^2 + y_2^2)^{3/2}}  d^2.$$
Meanwhile, the triangle inequality yields
$$ 2\|(0, c d^2) - y\| \leq 2\sqrt{y_1^2 + y_2^2} + 2 c d^2$$
and thus
$$  0 = I \geq  d^2 \int_{X} \left(  \frac{y_2^2}{((|y_1| + d)^2 + y_2^2)^{3/2}}  - 2c \right) ~ d\mu(y).$$
We also note, $\mu$ being a probability measure, that this is equivalent to
$$  \int_{X}   \frac{y_2^2}{((|y_1| + d)^2 + y_2^2)^{3/2}}  ~ d\mu(y) \leq 2c.$$
Since the estimate is uniformly true, we may take $d \rightarrow 0^+$ and deduce
$$  \int_{X}   \frac{y_2^2}{(y_1^2 + y_2^2)^{3/2}}  ~ d\mu(y) \leq 2c$$
and since $y_1^2 + y_2^2 \leq 1.001^2  \leq 1.003$, this would imply
$$  \int_{X}   y_2^2  ~ d\mu(y) \leq 3c.$$ 
It remains to give a lower bound for this integral. Suppose now that a lot of the mass of $\mu$ is close to the origin, for example
$$ \mu \left( \left\{y \in X:  y_2 \geq -0.1 \right\} \right) \geq \frac{4}{5}.$$
Then the distance integral will surely fluctuate quite a bit, since
$$  \int_{X} \| (0, c d^2) - y\| d\mu(y) \leq 0.2 \frac{4}{5} + 2 \frac{1}{5} \leq 0.6$$ 
while for the antipodal point $z$ on the opposite side of the curve
$$  \int_{X} \| z - y\| d\mu(y) \geq 1 \cdot \frac{4}{5} = 0.8$$
implying that the integral cannot be constant. Then, however,
$$  \int_{X}   y_2^2 ~d\mu(y) \geq \frac{1}{100} \frac{1}{5} = \frac{1}{500}.$$
This implies $c \geq 1/2500$ implying that the curvature has to be at least $1/10000$.
\end{proof}

\section{Proof of the Theorem}
\subsection{Main Idea.} \S 4.1 is dedicated to explaining the main idea behind the existence proof; the idea is versatile and can be implemented in other settings and in higher dimensions. However, it is also the only way we currently know how to prove existence of such a measure.
The idea is centered around an 1958 result of Bj\"orck \cite{bjorck}. Let $\Omega \subset \mathbb{R}^n$ be a compact set and consider, for every probability measure $\mu$ supported on $\Omega$, the functional
$$ I(\mu) = \int_{\Omega \times \Omega} \|x-y\| d\mu(x) d\mu(y).$$
Bj\"orck \cite{bjorck} considers the problem of maximizing the functional, i.e.
$$ \mu^* = \arg\max_{\mu \in \mathcal{P}(\Omega)} I(\mu),$$
 and proves that such a measure exists, that it is unique and and that its support is contained in the boundary.

\begin{thm}[Bj\"orck \cite{bjorck}] Let $\Omega \subset \mathbb{R}^n$ be compact. There exists a unique maximizing probability measure $\mu^*$ such that $I(\mu) \leq I(\mu^*)$ for all probability measure $\mu$ supported in $\Omega$. Moreover,
$$ \int_{\Omega} \|x-y\| d\mu(y) \qquad \qquad \begin{cases} = I(\mu^*) \qquad &\mbox{if}~x \in \emph{supp}(\mu^*) \\
\leq I(\mu^*) \qquad &\mbox{if}~x \notin  \emph{supp}(\mu^*). \end{cases}$$
Furthermore, $\mbox{supp}(\mu^*) \subseteq \partial \Omega$.
\end{thm}
\begin{proof}[Sketch of the Proof.] Since the argument is relatively simple, we include a sketch. Note that $I(\mu) \leq \mbox{diam}(\Omega) < \infty$. Take a sequence of probability measures $\mu_n$ so that $I(\mu_n)$ converges to the supremum, there exists a subsequence that converges weakly. Since $\|x - y\|$ is continuous, the limit is an extremizer. 
The statement concerning $ \int_{\Omega} \|x-y\| d\mu(y)$ follows from an Euler-Lagrange argument (see \cite{bjorck} for details) which then also implies uniqueness. The fact that $\mu$ is supported in the boundary then follows from the fact that, for fixed $y$, the function $\|x-y\|$ is strictly subharmonic and thus $\int_{\Omega} \|x-y\| d\mu(y)$ is  strictly subharmonic and thus cannot attain a maximal value in the interior of the domain.
\end{proof}

 This provides a sufficient condition for the existence of our desired measure.
 
 \begin{corollary} Let $\Omega \subset \mathbb{R}^n$ be compact. If the measure $\mu^*$ that maximizes 
 $$ I(\mu) = \int_{\Omega \times \Omega} \|x-y\| d\mu(x) d\mu(y)$$
  has dense support on $\partial \Omega$, then 
 $$  \int_{ \partial \Omega} \|x-y\| d\mu^*(y) \qquad \mbox{is a constant independent of}~x \in \partial \Omega.$$
 \end{corollary}
 \begin{proof}
This follows from the result of Bj\"orck since
$   \int_{ \partial \Omega} \|x-y\| d\mu^*(y)$ is continuous 
and for all $x \in \mbox{supp}(\mu^*)$ one has $f(x) = I(\mu^*)$.
\end{proof}

\subsection{Nearly round}
We assume the statement is false and for every $n \in \mathbb{N}$, there exists a smooth, closed, convex curve $X_n$ in $\mathbb{R}^2$ satisfying
$$1 - \frac{1}{n} \leq \frac{\mbox{length}(X_n)}{2\pi }  \min_{x \in X_n} \kappa(x)   \leq  \frac{\mbox{length}(X_n)}{2\pi }  \max_{x \in X_n} \kappa(x)  \leq 1 + \frac{1}{n}$$
for which no probability measure supported on $X_n$ solves the integral equation. Since the entire problem is invariant under scaling, we may assume without loss of generality that $X_n$ has length $2\pi$. We first prove that this sequence of (hypothetical) counterexamples has to be converge to the unit circle.

\begin{lemma}
Let $(x(s), y(s))$ is an arclength parametrization of $X_n$. There exists a rotation and translation of the curve $(x(s), y(s))$ so that
$$ |x(s) - \cos{(s)}| + |y(s) - \sin{(s)}| \leq \frac{16 \pi^2}{n}.$$
\end{lemma}
\begin{proof}
We use the Fundamental Theorem of Plane Curves (see, for example, \cite[Theorem 1.88]{alencar}) to deduce that the curve can be parametrized by, after possibly a rotation and a translation, by
\begin{align*}
x(s) = \int_0^s \cos\left( \int_0^{\tau} \kappa(\xi) d\xi \right) d\tau \quad \mbox{and} \quad y(s) = \int_0^s \sin\left( \int_0^{\tau} \kappa(\xi) d\xi \right) d\tau.
\end{align*}
If the curvature was $\kappa \equiv 1$, this would give rise to the parametrization $(\sin{s}, 1-\cos{s})$ which traces out a circle and we shall compare the curve to that particular circle. The uniform lower bound on the curvature implies
$$  \int_0^{\tau} \kappa(\xi) d\xi \geq  \left(1 - \frac{1}{n} \right) \tau.$$
As for the upper bound, we use that the total curvature is $2\pi$ and obtain
\begin{align*}
  \int_0^{\tau} \kappa(\xi) d\xi &=   \int_0^{2\pi} \kappa(\xi) d\xi  - \int_{\tau}^{2\pi} \kappa(\xi) d\xi \\
  &\leq 2\pi - (2\pi - \tau) \left(1 - \frac{1}{n}\right) \leq \frac{2\pi}{n} + \left(1 - \frac{1}{n} \right) \tau.
\end{align*}
Altogether, we have
$$ \left|  \int_0^{\tau} \kappa(\xi) d\xi - \tau \right| \leq \frac{2\pi + \tau}{n} \leq \frac{4\pi}{n}.$$
Therefore
\begin{align*}
\left| x(s) - \sin{(s)} \right| &= \left|  \int_0^s \cos\left( \int_0^{\tau} \kappa(\xi) d\xi \right) d\tau - \int_0^s \cos\left( \int_0^{\tau} 1 d\xi \right) d\tau \right| \\
&\leq  \left|  \int_0^s \cos\left( \int_0^{\tau} \kappa(\xi) d\xi \right) - \cos\left( \int_0^{\tau} 1 d\xi \right) d\tau \right| \\
&\leq  \left|  \int_0^s  \left|  \int_0^{\tau} \kappa(\xi) d\xi - \tau \right| d\tau \right| \leq  \left|  \int_0^s  \frac{4\pi}{n} d\tau \right| \leq \frac{8 \pi^2}{n}.
 \end{align*}
 The same is true for the second coordinate $y(s)$ and we obtain the desired result.
\end{proof}

\subsection{Proof of the Theorem}

  For any $X_n$, we use $\Omega_n$ to denote the enclosed domain and let 
$$ \mu_n^* = \arg\max_{\mu \in \mathcal{P}(\Omega_n)}  \int_{\Omega_n \times \Omega_n} \|x-y\| d\mu(x) d\mu(y).$$
By Bj\"orck's Theorem, the measure $\mu_n^*$ exists and is supported in $\partial \Omega_n$.
In light of Corollary 1, if the support of $\mu_n^*$ was dense, then a probability measure solving our integral equation would exist. We can therefore assume that the support is not dense, there exists an open set that does not intersect the support of the measure.  

\begin{center}
\begin{figure}[h!]
\begin{tikzpicture}
\filldraw (0,0) circle (0.06cm);
\filldraw (3,0) circle (0.06cm);
\draw [thick] (0,0) to[out=30, in=150] (3,0);
\node at (0, -0.3) {$(-d_n,0)$};
\node at (3, -0.3) {$(d_n,0)$};
\filldraw (1.5, 0.43) circle (0.06cm);
\node at (1.5, 0.8) {$(0,e_n)$};
\end{tikzpicture}
\caption{The local picture: $\mu_n^*$ has no support on the arc. The lower bound on the curvature implies a lower bound on $e_n$. }
\label{fig:thm1}
\end{figure}
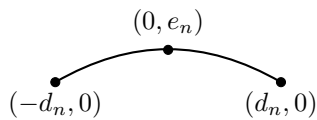
\end{center}

We extend the arc so that it is bounded by elements of the closure of the support of $\mu_n^*$ on both sides. We may assume w.l.o.g. after rotating and translating the domain $\Omega_n$ that these points are give by $(-d_n,0)$ and $(d_n,0)$ for some $d_n>0$.  We shall henceforth assume that all the curves $X_n$ have been rotated and translated so as to conform to this picture. We note that $e_n > 0$ as a consequence of the lower bound on the curvature; note, however, that the curve need not be symmetric around the $y-$axis and $e_n$ may not necessarily be the point on the curve with the largest $y-$coordinate.
Since $(-d_n, 0)$ and $(d_n,0)$ are in the support of the measure $\mu_n$, Bj\"orck's Theorem implies that
$$ \int_{X_n} \left\|y - (-d_n,0) \right\| d\mu_n^*(y) = I(\mu_n^*) = \int_{X_n} \left\|y - (d_n,0) \right\| d\mu_n^*(y)$$ 
and this is the maximal value that is being attained by the integral, hence
$$ \int_{X_n} \left\|y - (0,e_n) \right\| d\mu_n^*(y) \leq  I(\mu_n^*).$$
These two facts combined imply the inequality
$$ \int_{X_n} \left(2 \left\|y - (0,e_n) \right\| - \left\|y - (-d_n,0) \right\| - \left\|y - (d_n,0) \right\|\right) d\mu_n^*(y) \leq 0. \quad (\diamond)$$
The remainder of the argument is dedicated to deriving a contradiction to $(\diamond)$ assuming curvature to be sufficiently close to constant. It will helpful to think of $d_n$ as `nearly fixed'. Since $0 \leq d_n \leq \pi$, the sequence $(d_n)_{n=1}^{\infty}$ has a convergent subsequence which we again denote by $(d_n)_{n=1}^{\infty} \rightarrow d_*$. We distinguish the two cases: the case $d_* = 0$ and the case $d_* > 0$. \\

\textbf{Case 1} ($d_* > 0$)\textbf{.} This case is a little bit easier since everything happens at scale $\sim 1$. The curves converge to a circle in the Gromov-Hausdorff distance, this forces 
$$ e_n \rightarrow 1 - \sqrt{1-d_*^2}$$
Since $X_n$ converges to a circle containing $(\pm d_*, 0)$, it has to converge to the circle with center in $(0, -\sqrt{1-d_*^2})$. Abbreviating the three points in Fig. \ref{fig:thm1} as
$$ a = (-d_*, 0), b = (d_*, 0), c=(0, 1- \sqrt{1-d_*^2})$$
as well as the midpoint of the circle as $m = (0, -\sqrt{1-d_*^2})$ and the variable $z = (\cos{t}, \sin{t}) +m$,
the integrand in $(\diamond)$ will converge to
\begin{align*}
 2\|z-c\| - \|z-a\| - \|z-b\| &= \sqrt{8 - 8 \sin{(t)}} \\
 &- \sqrt{2 - 2d_* \cos{(t)} - 2\sqrt{1-d^2} \sin{(t)}} \\
 &- \sqrt{2 + 2d_* \cos{(t)} - \sqrt{1-d_*^2} \sin{(t)}}.
\end{align*}
This function, as a function of $t$, only has 2 roots. Since $f(\pm d_*,0) > 0$ and since, by assumption, the measure $\mu_n^*$ does not put any weight on the arc between $(\pm d_*, 0)$, the integrand is ultimately positive on the entire segment of the curve $X_n$ where the measure is supported. This contradicts $(\diamond)$ and concludes this case.\\

\textbf{Case 2} ($d_* = 0$)\textbf{.} This case is a little bit different because the integrand in $(\diamond)$ will converge to 0, a more refined type of analysis at length-scale $\sim d_n^2$ is required. We can write the convex curve between $(-d_n, 0)$ and $(d_n, 0)$ as the graph of a function $y = \phi(x)$. The lower bound on the curvature implies
$$ 1 - \frac{1}{n} \leq \kappa \leq \frac{|\phi''(x)|}{(1 + \phi'(x)^2)^{3/2}} \leq |\phi''(x)|.$$
Let us now assume $\phi'(0) \leq 0$ (the case $\phi'(0) \geq 0$ is analogous because the picture in Fig. \ref{fig:thm1} can be reflected around the $y-$axis). Then
\begin{align*}
 \phi(d_n) - \phi(0) &= \int_0^{d_n} \phi'(x) dx = \int_0^{d_n}  \left( \phi'(0) + \int_0^{x}  \phi''(y) dy \right) dx\\
 &\leq \int_0^{d_n}  \left(  \int_0^{x}  \phi''(y) dy \right) dx \\
 &\leq \int_0^{d_n}  \left(  \int_0^{x} -\left(1 - \frac{1}{n}\right) dy \right) dx = -\left( 1 - \frac{1}{n}\right) \frac{d_n^2}{2}.
 \end{align*}
 
 This implies a lower bound on $e_n$ as shown in Fig. \ref{fig:thm11}. For purely algebraic reasons, expressions that become easier to manipulate, we will approximate $e_n$ by $d_n^2/2$ (which, as $n \rightarrow \infty$, will be fairly accurate and only induce a small error). The triangle inequality together the lower bound on $e_n$ and $(\diamond)$ implies
 $$ \int_{X_n} \left(2 \left\|y - \left(0,\frac{d_n^2}{2}\right) \right\| - \left\|y - (-d_n,0) \right\| - \left\|y - (d_n,0) \right\|\right) d\mu_n^*(y) \leq \frac{d_n^2}{n}. \quad (\diamond \diamond)$$
 
 \begin{center}
\begin{figure}[h!]
\begin{tikzpicture}
\filldraw (0,0) circle (0.06cm);
\filldraw (3,0) circle (0.06cm);
\draw [thick] (0,0) to[out=30, in=150] (3,0);
\node at (0, -0.3) {$(-d_n,0)$};
\node at (3, -0.3) {$(d_n,0)$};
\filldraw (1.5, 0.43) circle (0.06cm);
\node at (1.5, 0.8) {$(0,e_n)$};
\node at (5, 0.5) {$e_n \geq \left(1 - \frac{1}{n}\right) \frac{d_n^2}{2}$};
\end{tikzpicture}
\caption{The local picture: $\mu_n^*$ has no support on the arc. The lower bound on the curvature implies a lower bound on $e_n$. }
\label{fig:thm11}
\end{figure}
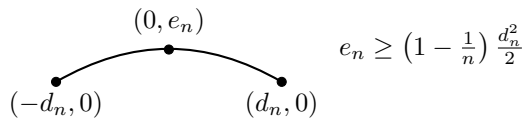
\end{center}

The integrand in $(\diamond \diamond)$ can be written
$$ f(y_1, y_2) = 2 \sqrt{y_1^2 + (y_2- d_n^2/2)^2} - \sqrt{(y_1 - d_n)^2 + y_2^2} - \sqrt{(y_1 + d_n)^2 + y_2^2}.$$
Our goal is to show that this integrand is `large' over the domain of integration where the measure $\mu_n^*$ is supported.  First observe that $f( \pm d_n,0) > 0$. The behavior of $f$ along the arc between these two points $(\pm d_n, 0)$ is more complicated but ultimately irrelevant because the measure $\mu_n^*$ has no support on that arc.

\begin{center}
\begin{figure}[h!]
\begin{tikzpicture}
\node at (0,0) {\includegraphics[width=0.4\textwidth]{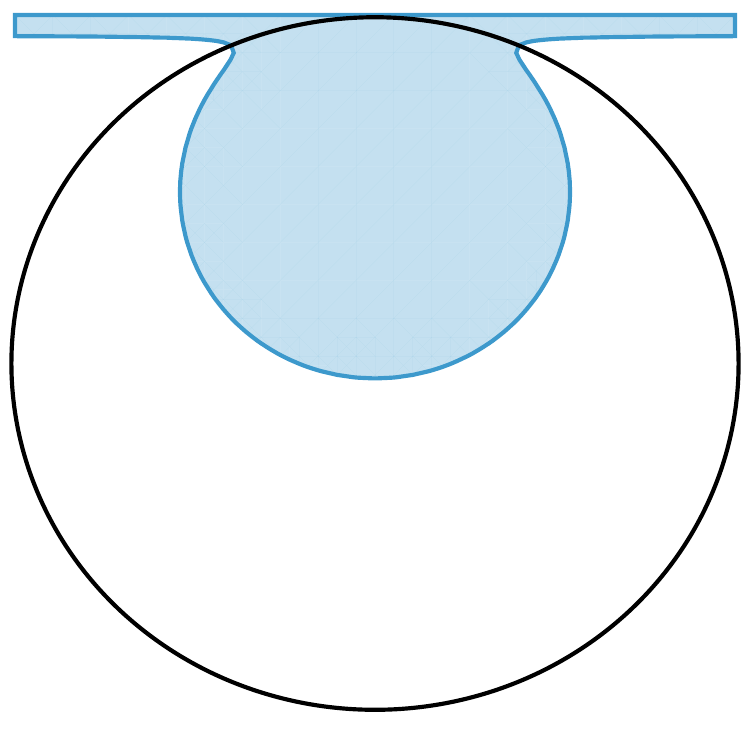}};
\end{tikzpicture}
\vspace{-10pt}
\caption{The set $\left\{(y_1, y_2) \in \mathbb{R}^2: f(y_1, y_2) \leq 0 \right\}$ when $d_n=0.4$ (blue) shown against a disk with radius 1 going through the two points $(-d_n,0)$ and $(d_n,0)$.}
\label{fig:example}
\end{figure}
\end{center}

A very precise analysis of this case is possible: in particular, if $(y_1, y_2)$ is a root in the lower half-plane, $f(y_1, y_2) = 0$ and $y_2 \leq 0$, then
$$ y_1 = \pm \frac{\sqrt{4-d_n^2 + 4y_2}}{\sqrt{4d_n^2 - 16 y_2}} (d_n^2 - 4 y_2) \qquad \mbox{for all} \quad \frac{d_n^2}{4} - 1 \leq y_2 \leq 0.$$
Geometrically, this means that the `blue bubble' in Fig. \ref{fig:example} has a parametrization in closed form that one could work with. We have found this to be algebraically cumbersome and will use a slightly different argument, however, it is worth mentioning that a more precise analysis of the blue bubble might lead to better quantitative estimates. As for our different argument, we first note that for $(y_1, y_2)$ at a strict positive distance from the origin, we have the asymptotic expansion as $d_n \rightarrow 0$
$$ f(y_1, y_2) = \frac{ (-y_2) y_1^2 - y_2^2 - y_2^3}{(y_1^2 + y_2^2)^{3/2}} d_n^2 + \mathcal{O}(d_n^4).$$
It is easy to see that the leading coefficient is positive outside a disk of radius $1/2$ centered at $(0, -1/2)$. In particular, this shows that for any $\varepsilon \rightarrow 0$, there exists $\delta > 0$ so that for all $n \in \mathbb{N}$ sufficiently large
$$ \int_{X_n} f(y) 1_{\|y\| \geq \delta} ~d\mu_n(y) \geq (2 - \varepsilon) d_n^2.$$
This is close to contradicting $(\diamond \diamond)$. It remains to show the part of the curve close to $(0,0)$, where the error term in the series expansion is not uniform, cannot contribute too much. We introduce the function $g(y_1, y_2)$ which will be an algebraically simplified version of $f(y_1, y_2)$
$$ g(y_1, y_2) =  2 \sqrt{y_1^2 + y_2^2} - \sqrt{(y_1 - d_n)^2 + y_2^2} - \sqrt{(y_1 + d_n)^2 + y_2^2}$$
and note that the triangle inequality implies $|f(y_1, y_2) - g(y_1,y_2)| \leq d_n^2$ which is at the same scale as our lower bound but sufficiently close to be used in a small neighborhood. Moreover, for $y_2 \leq 0$, we have
$ f(y_1, y_2) \geq g(y_1, y_2).$
The reason for working with $g(y_1, y_2)$ is that certain monotonicity properties become more readily apparent: if $y_1 \geq 0$ and $y_2 \leq 0$, then the function is monotonically increasing in $y_1$. This can be seen as follows.
 We have
$$ \frac{\partial g}{\partial y_1} = \frac{2 y_1}{\sqrt{y_1^2 + y_2^2}} - \frac{y_1 + d_n}{\sqrt{(d_n+y_1)^2 + y_2^2}} - \frac{y_1 - d_n}{\sqrt{(y_1 - d_n)^2 + y_2^2}}.$$
Thinking of $y_2$ as a constant and writing 
 $$ h(t) = \frac{t}{\sqrt{t^2 + y_2^2}} \qquad \mbox{we have} \qquad h''(t) = -\frac{3t}{(1+t^2)^{5/2}}$$
and can rewrite the derivative as
$$ \frac{\partial g}{\partial y_1} = 2 h(y_1) - h(y_1 + d_n) - h(y_1 - d_n).$$
This means that for $y_1 \geq d_n$, all three arguments are positive and the function $h$ is concave in that region which implies
$$  \frac{\partial g}{\partial y_1}(y_1, y_2) \geq 0 \qquad \mbox{provided} \qquad y_1 \geq d_n.$$
The upper bound on the curvature implies that, for $d_n \ll 1/1000$, the curve segment
$$ X \cap \left\{(y_1, y_2): d_n \leq y_1 \leq \frac{1}{100}  \wedge y_2 \leq 0 \right\}$$
lies above the graph $y = -5(x-d_n)^2$. Using the relative rigidity of the arclength measure of a convex curve, we note that the missing integral can be bounded from below by
\begin{align*}
 \int_{X_n} f(y) 1_{\|y\| \leq \delta} ~d\mu_n(y) \geq 10 \int_{d_n}^{1/100} g(x, -5 (x-d_n)^2) dx.
\end{align*}
However, this function allows for a uniformly convergent Taylor expansion and
$$ g(x, -5 (x - d_n)^2) = - \frac{25 \sqrt{x^2 + 25 x^4}}{(1+25 x^2)^2}d^2 + \mathcal{O}(d^3)$$
which shows that the contribution of the missing part of the integral is, asymptotically, larger than $-d_n^2/2$ which contradicts ($\diamond \diamond)$ and concludes the proof. $\qed$\\

\textbf{Remark.} It is clear that this approach is very far from optimal with respect to constants. The argument does suggest a natural limit for what one might hope for in the case of curves that are close to the circle (which, one might argue, might be the nicest examples). Assuming the local structure to be
 \begin{center}
\begin{figure}[h!]
\begin{tikzpicture}
\filldraw (0,0) circle (0.06cm);
\filldraw (3,0) circle (0.06cm);
\draw [thick] (0,0) to[out=30, in=150] (3,0);
\node at (0, -0.3) {$(-d_n,0)$};
\node at (3, -0.3) {$(d_n,0)$};
\filldraw (1.5, 0.43) circle (0.06cm);
\node at (1.5, 0.8) {$(0,c \cdot d_n^2)$};
\end{tikzpicture}
 \end{figure}
\end{center}
the relevant question is whether it is possible to have
 $$ \int_{X_n} \left(2 \left\|y - \left(0, c \cdot d_n^2 \right) \right\| - \left\|y - (-d_n,0) \right\| - \left\|y - (d_n,0) \right\|\right) d\mu_n^*(y)  \sim 0$$
A Taylor series expansion of the integrand shows that
$$ \mbox{integrand} = \frac{ -2 c y_1^2y_2 - 2c  y_2^3- y_2^2}{\left(y_1^2+y_2^2\right)^{3/2}} d^2 + o(d^2).$$
Plugging in $y_1 = \cos{t}$ and $y_2 = -1 + \sin{t}$, the numerator simplifies to
$$ -2 c y_1^2y_2 - 2c  y_2^3- y_2^2 = (-1 + 4c) (1 - \sin{t})^2$$
which shows a critical threshold for $c=1/4$ corresponding to curvature $1/2$. This reiterates the underlying theme that existence and nonexistence will be governed by critical thresholds involving the curvature.

\section{Curvature as `almost' a solution}

Existence and non-existence of a probability measure satisfying the integral equation is tied to curvature. More is true: the probability measure is itself tied to the curvature of the underlying curve (see, Fig. \ref{fig:2} and Fig. \ref{fig:3} for example, see Fig. \ref{fig:4} for quantitative comparison). We show that for curves that are close to the unit circle, the curvature function is a measure that \textit{nearly} solves the integral equation. 
The statement is most easily stated using the support function characterization (see Fig. \ref{fig:supp}).
 We introduce the support function $h(\theta)$ as the projection of 
$$ h(\theta) = \max_{x \in X} \left\langle \begin{pmatrix} \cos{\theta} \\ \sin{\theta} \end{pmatrix}, x \right\rangle.$$
The circle has the support function parametrization $h(\theta) = 1$. 

\begin{center}
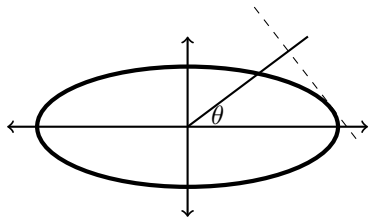
\begin{figure}[h!]
\begin{tikzpicture}[scale=0.8]
\draw [thick, <->] (-3,0) -- (3,0);
\draw [thick, <->] (0,-1.5) -- (0, 1.5);
\draw [ultra thick] (0,0) ellipse (2.5cm and 1cm);
\draw [thick] (0,0) -- (2,1.5);
\draw [dashed] (2.8, -0.2) -- (1.1, 2);
\node at (0.5, 0.2) {$\theta$};
\end{tikzpicture}
\vspace{-10pt}
\caption{The support function parametrization.}
\label{fig:supp}
\end{figure}
\end{center}

\begin{proposition} Let $X$ be a closed, smooth convex curve with curvature $\kappa$ bounded from below. Then the measure
$\mu = \kappa~ d\sigma$ may not solve the integral equation but the deviation from a constant
$$ V = \max_{x \in X} \int_{X} \| x- y\| d\mu(y) -  \min_{x \in X} \int_{X} \| x- y\| d\mu(y)$$
is bounded from above by
$$ V \leq  4 \max_{0 \leq t \leq 1} |h(t) - 1| + |h'(t)|.$$
\end{proposition}

Curvature in the support function characterization is given by
$ \kappa(t) = 1/(h(t) + h''(t))$
which depends on the \textit{second} derivative of $h$ while the error term in the Theorem only depends on deviation from a constant and the first derivative. The proof follows from a suitable cancellation of the arclength element with curvature which becomes most readily apparent in the support function parametrization. Proposition 3 might not give the complete picture, there might be other results that make the empirical observation illustrated in Fig. \ref{fig:4} rigorous.

 \begin{center}
\begin{figure}[h!]
\begin{tikzpicture}
\node at (0,0) {\includegraphics[width=0.25\textwidth]{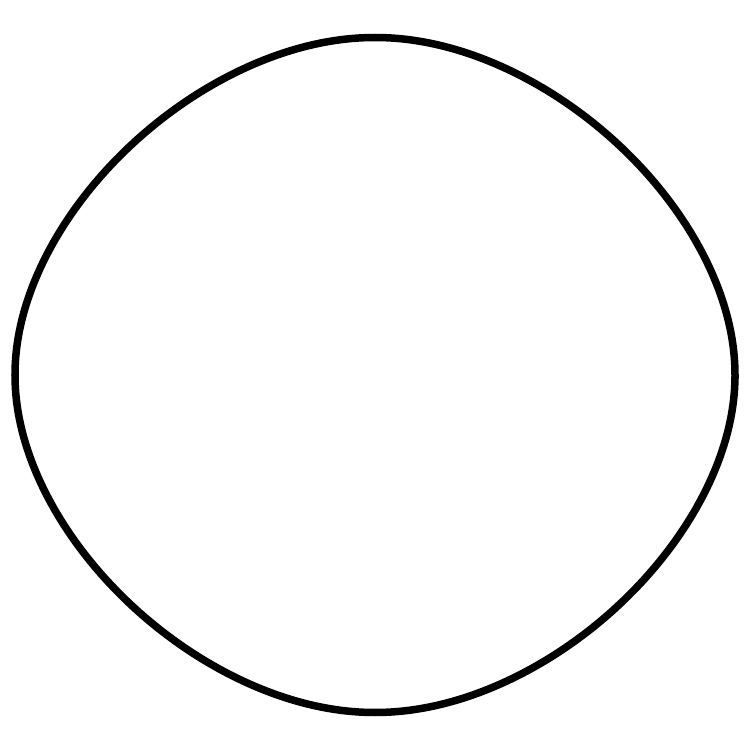}};
\node at (3.8,0) {\includegraphics[width=0.25\textwidth]{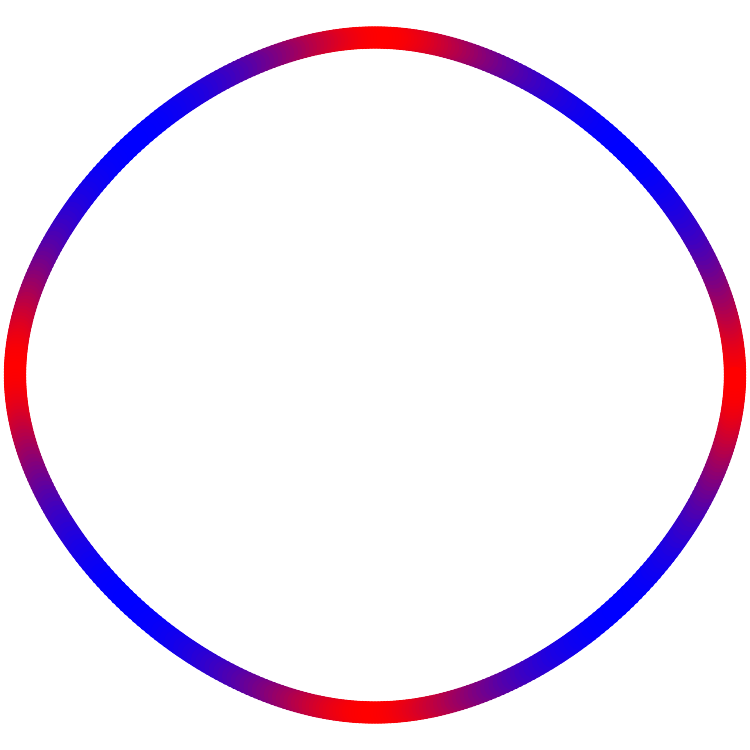}};
\node at (8,0) {\includegraphics[width=0.35\textwidth]{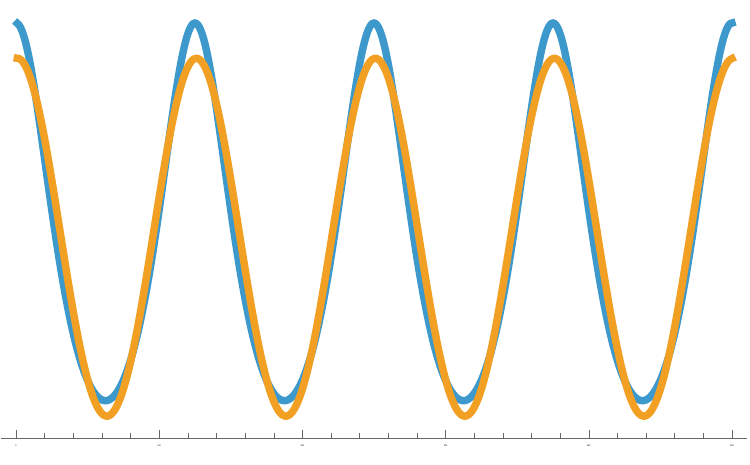}};
\end{tikzpicture}
\caption{A curve (left), colored by density of the measure (middle) and the (rescaled) density compared to the curvature both shown as a function of arclength (right).}
\label{fig:4}
\end{figure}
\end{center}

\begin{proof}[Proof of Proposition 3] The support function $h(\theta)$
$$ h(\theta) = \max_{x \in X} \left\langle \begin{pmatrix} \cos{\theta} \\ \sin{\theta} \end{pmatrix}, x \right\rangle$$
 can be rewritten into the more classical $(x(\theta), y(\theta))$ by differentiation 
 \begin{align*}
x(t) &= h(t) \cos{(t)} - h'(t) \sin{(t)} \\
y(t) &= h(t) \sin{(t)} + h'(t) \cos{(t)}.
\end{align*}
This is \textit{not} an arc-length parametrization, a short computation shows that 
\begin{align*}
x'(t) &= h'(t) \cos{t} - h(t) \sin{t} - h''(t) \sin{t} - h'(t) \cos{t} \\
y'(t) &= h'(t) \sin{t} + h(t) \cos{t} + h''(t) \cos{t} - h'(t) \sin{t}
\end{align*}
which implies that the arc-length is given by
$$ \sqrt{x'(t)^2 + y'(t)^2} = h(t) + h''(t).$$
Another computation shows that the curvature is given by
$$ \kappa(t) = \frac{x'(t) y''(t) - y'(t) x''(t)}{(x'(t)^2 + y'(t)^2)^{3/2}} = \frac{1}{h(t) + h''(t)}.$$
Parametrizing the curve as $\gamma(t) = (x(t), y(t))$ and 
accounting for the arc-length parametrization,
the integral equation becomes 
$$ \int_{0}^{2\pi} \| \gamma(t)- \gamma(s)\|  (h(s) + h''(s))   d\mu(s) = c$$
for some constant $c$ independent of $t$. Making the ansatz
$$ \mu(s) = \kappa(s) ds = \frac{ds}{h(s) + h''(s)},$$
we arrive at the function
$$ V(t) =  \int_{0}^{2\pi} \| \gamma(t)- \gamma(s)\| ds.$$
Since 
$ \gamma(t) = (h(t) \cos{(t)} - h'(t) \sin{(t)}, h(t) \sin{(t)} + h'(t) \cos{(t)}),$
we have $  \| \gamma(t) - (\cos{t}, \sin{t})\| \leq |h(t) - 1| + |h'(t)|$. Hence,
\begin{align*}
V(t) &=  \int_{0}^{2\pi} \| \gamma(t)- \gamma(s)\| ds \\
&= \int_{0}^{2\pi} \| (\cos(t), \sin(t)) - (\cos(s), \sin(s))\| ds \\
& +  \int_{0}^{2\pi} |h(t) - 1| + |h'(t)| + |h(s) - 1| + |h'(s)| ds. 
\end{align*}
Hence
$$ \left| V(t) - 4 \right| \leq 4 \pi \max_{0 \leq t \leq 1} |h(t) - 1| + |h'(t)|.$$
Recalling that we are looking for a probability measure and that $\int_0^{2\pi} \kappa(s) ds = 2\pi$, we can rescale by $2\pi$ and get the desired result.
\end{proof}

\end{document}